\documentclass{article}
\usepackage[utf8]{inputenc}
\usepackage{amsfonts}
\usepackage{enumerate}
\usepackage[british]{babel}
\usepackage{amsthm}
\usepackage{amsmath}
\usepackage{amssymb}
\usepackage{caption}
\usepackage{stmaryrd}
\usepackage[toc,page]{appendix}
\usepackage{empheq,mathtools}
\usepackage{accents}
\usepackage[utf8]{inputenc}

\title{Gelfand Pairs of Complex Reflection Groups}
\author{Robin van Haastrecht}

\newtheorem{theorem}{Theorem}[section]

\theoremstyle{definition}
\newtheorem{definition}[theorem]{Definition}

\newtheorem*{remark}{Remark}

\theoremstyle{theorem}
\newtheorem{proposition}[theorem]{Proposition}
\newtheorem{lemma}[theorem]{Lemma}
\newtheorem{corollary}[theorem]{Corollary}

\newcommand\myeq{\stackrel{\mathclap{\normalfont\mbox{(1)}}}{=}}

\begin{document}
\thispagestyle{empty}
\begin{centering} 

\textbf{\Large Spherical functions for Gelfand pairs of complex reflection groups}\\[0.2\baselineskip] 
Robin van Haastrecht \\

\vspace{1.5cm}
\begin{abstract}
	In this article the zonal spherical functions of the Gelfand pair $(G(r,d,n), S_n)$ of complex reflection groups will be calculated. After this, a product formula for these spherical functions and a discrete analog of the Laplace operator which has the spherical functions as eigenfunctions will be given.
\end{abstract}
\end{centering}
{\small \textit{Keywords:} Gelfand pairs of finite groups; Finite complex reflection groups; $\\ (n+1, m+1)$-hypergeometric functions;  Zonal spherical functions; Product formulas; Laplace operator}

\section{Introduction}
If $G$ and $H$ are groups and the induced representation $1_H^G$ splits multiplicity-free, $(G,H)$ is a Gelfand pair. The theory of Gelfand pairs was originally started in the setting of Lie groups in a paper by I.M. Gelfand \cite{gelfandorigin}. Since then, Gelfand pairs of locally compact groups and finite groups have also been studied, and have applications in areas such as special functions \cite{akamizu, dunklpol, mizukawa04} and probability \cite{cecsil08}. In this article we will study finite Gelfand pairs of complex reflection groups. The group theoretic approach of Gelfand pairs will be used to give some results of special functions as spherical functions of complex reflection groups and the spherical functions of more general complex reflection groups will be given. An interesting follow up question would be if we could quantise these spherical functions like in \cite{koelink}.

In Section 2 the preliminaries for the article will be given; we recall the theory of finite Gelfand pairs and spherical functions and we remind ourselves of the complex reflection groups $G(r,d,n)$. The results follow after. In Section 3 the spherical functions of the complex reflection groups $G(r,d,n)$ will be given. Sections 4 and 5 will focus on an application of the theory of Gelfand pairs to special functions, respectively a product formula for the hypergeometric function and a discrete analogue of the Laplace operator which has the spherical functions as eigenfunctions.

\vspace{.5 cm}

\textit{Notation}

In this article $\mathbb{N}$ is defined to be the natural numbers including $0$. The binomial $\binom{n}{k}$ and multinomial $\binom{n}{k_1, \dots, k_n}$ are zero when $k < 0$ or a $k_i < 0$ respectively. The Pochhammer symbol is defined as: $(x)_n = x(x+1)\dots(x+n-1)$ where $(x)_0 = 1$. This gives us $\binom{x}{m} = (-1)^m \frac{(-x)_m}{m!}$. The groups considered in this article will be assumed to be finite. The inner product spaces $V$ considered here are complex. The dihedral group $D_r$ is of order $2r$.

\section{Preliminaries}

\subsection{Gelfand pairs and zonal spherical functions}
We will introduce Gelfand pairs and zonal spherical functions. In this section $G$ will be a finite group and $H \subseteq G$ a subgroup. $C(G)$ will be the $\mathbb{C}$-vector space of functions from $G$ to $\mathbb{C}$. This vector space comes equipped with a natural inner product, i.e. $\langle f,h \rangle = \sum_{g \in G} f(g)\overline{h(g)}$. Similarly, $C(G/H)$ is the vector space of functions on the space $G/H$, which can be identified with the subspace of right $H$-invariant functions of $C(G)$. The vector space of functions on $C(H \backslash G / H)$ can be similarly identified with the subspace of bi-$H$-invariant functions of $C(G)$. The space $C(G)$ can be turned into an algebra by using the convolution product:
$$(f*h)(g) = \sum_{x, y \ : \ xy = g}f(x)h(y) = \sum_{t \in G} f(t)h(t^{-1}g).$$
Note that this makes $C(G)$ isomorphic to the group algebra and both $C(G / H)$ and \\ $C(H \backslash G / H)$ are subalgebras as subspaces of $C(G)$.

\begin{definition}
A group $G$ with a subgroup $H$ form a \emph{Gelfand pair} $(G,H)$ if the representation $1_H^G = \mathrm{Ind}_H^G(1_H) = \bigoplus_{i=1}^s V_i$ is multiplicity-free.
\end{definition}

We have the equality $s = |H \backslash G / H|$ from \cite[Corollary 4.4.3]{cecsil08}. We will state a useful criterium for Gelfand pairs.

\begin{lemma}[Gelfand's Lemma {\cite[Exercise 4.3.3]{cecsil08}}]
Let $G$ be a group and $H \subseteq G$ a subgroup. If there exists an automorphism $\tau$ of $G$ such that $g^{-1} \in H\tau(g)H$, then $(G,H)$ is a Gelfand pair. We call such a Gelfand pair weakly symmetric. \qed
\label{11}
\end{lemma}

Let $(G,H)$ be a Gelfand pair, so $1_H^G$ splits multiplicity-free as a $\mathbb{C}[G]$-module, i.e. $1_H^G = \bigoplus_{i=1}^s V_i$. This means by Frobenius reciprocity that $\langle 1_H,\mathrm{Res}V_i \rangle_H = \langle 1_H^G,V_i \rangle_G = 1$. Thus every $V_i$ has a 1-dimensional subspace $V_i^H$ of $H$-invariant vectors. Remember each $V_i$ comes equipped with a $G$-invariant inner product. Now choose a unit vector $v_1^i \in V_i^H \subseteq V_i$. Note that in each case $\{v_1^i\}$ can be extended to an orthonormal basis $\{v_j^i\}_{j=1}^{\mathrm{dim}(V_i)}$ for $V_i$.  Let $C(G/H) \subseteq C(G)$ and define a linear map:
$$\phi_i : V_i \rightarrow C(G/H)$$
$$v \mapsto (g \mapsto \langle v | gv_1^i \rangle).$$
This map $\phi_i$ is in fact a $\mathbb{C}[G]$-homomorphism because:
$$\phi_i(gv)(k) = \langle gv,kv_1^i \rangle = \langle v,g^{-1} k v_1^i \rangle = \phi_i(v)(g^{-1}k) = (g(\phi_i(v))(k).$$
Note that $\phi_i$ is injective. It follows that $V_i \cong \phi_i(V_i)$ and we deduce: $C(G/H) = \bigoplus_{i=1}^s \phi_i(V_i)$. We are now ready to define the zonal spherical functions.

\begin{definition}
The \emph{zonal spherical functions} $\omega_i$ are the $\omega_i = \phi_i(v_1^i)$ for $1 \leq i \leq |H \backslash G / H|$. This means that $\omega_i(g) = \langle v_1^i | gv_1^i \rangle = \overline{\rho_{11}^i(g)}$.
\label{defspher}
\end{definition}

We list some properties of zonal spherical functions.

\begin{proposition}[{\cite[Proposition 4.5.7]{cecsil08}}]
The zonal spherical functions have the following properties:

(1) $\omega_i(h_1gh_2) = \omega_i(g)$ for any $g \in G$ and $h_1, h_2 \in H$.

(2) $\omega_i(e) = 1$ and $\omega_i(g^{-1}) = \overline{\omega_i(g)}$ for any $g \in G$.

(3) $\omega_i * \omega_j = \frac{|G|}{\mathrm{dim}(V_i)} \delta_{ij} \omega_i$.
\label{12} \qed
\end{proposition}

\begin{corollary}
Let $D_k$ be a double coset in $H \backslash G / H$, which can be indexed by $k$ ranging from $1$ to $s$. It follows that if we write $\omega_i(D_k) = \omega_i(g)$ for any $g \in D_k$, then:
$$\frac{1}{|G|} \sum_{k=1}^s |D_k| \omega_i(D_k)\overline{\omega_j(D_k)} = \delta_{ij} \mathrm{dim} (V_i)^{-1}.$$
Note that this means the spherical functions are orthogonal for the inner product of $C(G)$. Furthermore, the $\{\omega_i\}_{i=1}^s$ form an orthogonal basis for the subspace $C(H \backslash G / H) \subseteq C(G)$.
\label{13}
\end{corollary}

\begin{proof}
This follows by filling in $\omega_i * \omega_j(e)$ and using Proposition \ref{12}.
\end{proof}

\begin{remark}
Because $C(G/H) = \bigoplus_{i=1}^s \phi_i(V_i)$ and we have an orthonormal basis $\{v_j^i\}_{j=1}^{\mathrm{dim}(V_i)}$ for each $V_i$, we know that $\{ \phi_i(v_j^i)\}_{i,j}$ forms a basis for $C(G/H) \subseteq C(G)$. We know that $\phi_i(v_j^i)(g) = \langle v_j^i|gv_1^i \rangle = \overline{\rho^i_{j1}(g)}$.
\end{remark}

We have another characterisation of spherical functions.

\begin{theorem}[{\cite[Theorem 4.5.3]{cecsil08}}]
Let $(G,H)$ be a Gelfand pair. A bi-$H$-invariant non-zero function $\omega$ is a zonal spherical function if and only if:
\begin{equation}
\forall_{g,k \in G} \ \frac{1}{|H|}\sum_{h \in H} \omega(ghk) = \omega(g)\omega(k). \qed
\label{eq:sumformula}
\end{equation}
\label{15}
\end{theorem}

\subsection{Complex reflection groups}
In this section we will introduce the complex reflection groups $G(r,d,n)$. More information on these groups can be found in Shephard and Todd's article \cite{sheptodd}. We begin by defining $G(r,1,n)$. From now on $\xi = \exp(\frac{2 \pi i}{r})$ and $C_r = \langle \xi \rangle \cong \mathbb{Z} / r \mathbb{Z}$.

\begin{definition}
$G(r,1,n) = C_r \wr S_n : = C_r^n \rtimes S_n$ where if $\sigma \in S_n$ and $(\xi_1, \dots, \xi_n) \in C_r^n$ then $\sigma (\xi_1, \dots, \xi_n) = (\xi_{\sigma^{-1}(1)}, \dots, \xi_{\sigma^{-1}(n)})$.
\end{definition}

\begin{remark}
This group can be represented by the set of monomial matrices that have $r$-th roots of unity as entries. In the case of $r = 2$ the group $G(2,1,n) = H_n$ is the hyperoctahedral group, the Weyl group of type $B_n$ \cite{humphreys}.
\end{remark}

We see that $|G(r,1,n)| = r^nn!$. We now move on to the definition of $G(r,d,n)$ which is defined when $d | r$ and can be realised as a subgroup of $G(r,1,n)$. We set $p = \frac{r}{d}$ and remark that $\langle \xi^p \rangle = \langle \zeta \rangle = C_d$ for $\zeta = \exp(\frac{2 \pi i}d)$.

\begin{definition}
If $\Phi_r^d$ denotes the surjective group homomorphism:
$$\Phi_r^d: G(r,1,n) \rightarrow C_d$$
$$(\xi_1, \dots, \xi_n, \sigma) \mapsto (\xi_1 \dots \xi_n)^p$$
for $d|r$, the group $G(r,d,n)$ is defined as $\mathrm{Ker}(\Phi_r^d)$.
\end{definition}

\begin{remark}
A more direct realisation is given by:
$$G(r,d,n) = \{(\xi^{a_1}, \dots, \xi^{a_n}, \sigma) \in G(r,1,n) \ | \ a_1 + \dots + a_n = 0 \ \mathrm{mod} \ d\}.$$
\end{remark}

Important examples of these groups include $G(2,2,n)$, the Weyl group of type $D_n$, and $G(r,r,2)$, which is the dihedral group of order $2r$ as well as the Weyl group of type $I_2(r)$ \cite{humphreys}. We remark that $G(r,d,n)$ is normal and: 

$$G(r,1,n) / G(r,d,n) \cong C_d.$$

From here it follows that $|G(r,d,n)| = \frac{r^n}{d}n!$. We have $S_n \subseteq G(r,d,n)$ as a subgroup if we set $S_n = \{(1, \dots, 1, \sigma) \in G(r,d,n) \ | \ \sigma \in S_n \}$. We will study $(G(r,d,n), S_n)$, which will turn out to be a Gelfand pair. First we obtain some results about the representatives of the left cosets and double cosets of  $G(r,d,n)$ by $S_n$, contained in $G(r,d,n) / S_n$ and $S_n \backslash G(r,d,n) / S_n$.

\begin{proposition}
\mbox{}

(1) A complete set of representatives for the left cosets of $S_n$ in $G(r,d,n)$ is given by $\{(\xi_1, \dots, \xi_n, id) \in G(r,d,n)\}$.

(2) A complete set of representatives for the double cosets of $S_n$ in $G(r,d,n)$ is given by:
$$\{(\underbrace{1, \dots, 1}_{\mathit{l}_0}, \dots, \underbrace{\xi^{r-1}, \dots, \xi^{r-1}}_{\mathit{l}_{r-1}}, id) \in G(r,d,n) \ | \ \sum_{i = 0}^{r-1} \mathit{l}_i = n, \sum_{i = 0}^{r-1} i \mathit{l}_i = 0 \ \mathrm{mod} \ d \}$$
These representatives can be represented by an $r$-tuple $(\mathit{l}_0, \dots, \mathit{l}_{r-1})$ such that $\sum_{i=0}^{r-1} \mathit{l}_i = n$ and $\sum_{i = 0}^n i\mathit{l}_i \equiv 0 \ \mathrm{mod} \ d$.
\label{16}
\end{proposition}

\begin{proof}
(1): This can be seen by remarking that $(\xi_1, \dots, \xi_n, \sigma) \cdot \tau = (\xi_1, \dots, \xi_n, \sigma \tau)$. This means that $(\xi_1, \dots, \xi_n, \sigma)S_n = \{(\xi_1, \dots, \xi_n, \tau) \ | \ \tau \in S_n\}$.

(2): A left coset is described by a $(\xi_1, \dots, \xi_n, id) \in G(r,d,n)$. Note that a double coset is completely determined by the number of $\xi^i$ for $0 \leq i \leq r-1$ in an element of the coset. We call the number of times $\xi^i$ appears $\mathit{l}_i$. We obtain $\sum_{i = 0}^{r-1} \mathit{l}_i = n$. Because $(\xi^{a_1}, \dots, \xi^{a_n}, \sigma) \in G(r,d,n)$ it is also necessary that $\sum_{i = 0}^{r-1} i \mathit{l}_i = 0 \ \mathrm{mod} \ d$.
\end{proof}

\begin{proposition}
$(G(r,d,n), S_n)$ is a (weakly symmetric) Gelfand pair.
\label{17}
\end{proposition}

\begin{proof}
$(G(r,d,n), S_n)$ is a weakly symmetric Gelfand pair if there is a group automorphism $f$ such that $g^{-1} \in S_n f(g) S_n$ by Lemma \ref{11}.
Consider the group automorphism:
$$f: G(r,1,n) \rightarrow G(r,1,n)$$
$$(\xi_1, \dots, \xi_n, \sigma) \mapsto (\xi_1^{-1}, \dots, \xi_n^{-1}, \sigma).$$
By restriction to $G(r,d,n)$ we find a group automorphism of $G(r,d,n)$. Notice that:
$$(\xi^{a_1}, \dots, \xi^{a_n}, \sigma)^{-1} = (\xi^{-a_{\sigma(1)}}, \dots, \xi^{-a_{\sigma(n)}}, \sigma^{-1}) \in S_n f((\xi^{a_1}, \dots, \xi^{a_n}, \sigma)) S_n.$$
Hence the condition of Lemma \ref{11} is satisfied and the proposition is proven.
\end{proof}

We are now ready to state the irreducible representations making up $1_{S_n}^{G(r,1,n)}$ as stated by Mizukawa \cite{mizukawa04}. The group $G(r,d,n)$ acts on the ring of polynomials in $n$ complex variables as:
$$(\xi_1, \dots, \xi_n, \sigma)f(x_1, \dots, x_n) = f(\xi_{\sigma(1)}^{-1}x_{\sigma(1)}, \dots, \xi_{\sigma(n)}^{-1}x_{\sigma(n)}).$$
There is a map from $\mathbb{N}^r$ to the set of partitions:
$$\psi: \mathbb{N}^r \ni (\mathit{k}_0, \dots, \mathit{k}_{r-1}) \mapsto (\underbrace{0, \dots, 0}_{\mathit{k}_0}, \dots, \underbrace{r-1, \dots, r-1}_{\mathit{k}_{r-1}}).$$

\begin{proposition}[{\cite[Proposition 3.2]{mizukawa04}}]
The induced representation $1_{S_n}^{G(r,1,n)}$ is decomposed as:
$$1_{S_n}^{G(r,1,n)} \cong \bigoplus_{\{(\mathit{k}_0, \dots, \mathit{k}_{r-1}) \ | \ \sum_{i=0}^{r-1} \mathit{k}_i = n \}} V^{(\mathit{k}_0, \dots, \mathit{k}_{r-1})}.$$
Here $V^{(\mathit{k}_0, \dots, \mathit{k}_{r-1})}$ is an irreducible $G(r,1,n)$-module realised as:
$$V^{(\mathit{k}_0, \dots, \mathit{k}_{r-1})} = \bigoplus_{f \in M_n(\psi(\mathit{k}_0, \dots, \mathit{k}_{r-1}))} \mathbb{C}f.$$
Here $M_n(\lambda) = \{x_{\sigma(1)}^{\lambda_1} \dots x_{\sigma(n)}^{\lambda_n} \ | \ \sigma \in S_n \}$ for $\lambda = (\lambda_1, \dots, \lambda_n)$. \qed
\label{18}
\end{proposition}

Before we can give the zonal spherical functions of $(G(r,1,n), S_n)$ we need to define a generalised class of hypergeometric functions, originally introduced by Aomoto and Gelfand.

\begin{definition}
An $(n+1,m+1)$-hypergeometric functions is defined as:
$$F(\alpha, \beta, \gamma, X) = \\ \sum_{(a_{ij}) \in M_{n,m-n-1}(\mathbb{N})}\frac{\prod_{i=1}^n (\alpha_i)_{\sum_{j=1}^{m-n-1}a_{ij}} \prod_{i=1}^{m-n-1}(\beta_i)_{\sum_{j=1}^n a_{ji}}}{(\gamma)_{\sum_{i,j} a_{ij}}} \frac{\prod_{i,j} X_{ij}^{a_{ij}}}{\prod_{i,j} a_{ij}!}.$$
\label{defhgf}
\end{definition}

Here $X \in M_{n,m-n-1}(\mathbb{C})$, $\alpha$ is an $n$-tuple, $\beta$ is an $(m-n-1)$-tuple and $\gamma \in \mathbb{C}$. We adopt the convention that $\gamma$ is a negative integer and we sum over the $(a_{ij})_{i,j}$ such that $\\ \sum_{i,j} a_{ij} \leq -\gamma$, so the above sum is finite and we will not have to worry about convergence issues. Note that $(2,4)$-hypergeometric functions give the usual Gauss hypergeometric functions ${}_2 F_1$. We are now ready to give the spherical functions for the Gelfand pair $(G(r,1,n),S_n)$, first calculated by Mizukawa \cite{mizukawa04}. They are parametrised by the $(\mathit{k}_0, \dots, \mathit{k}_{r-1}) \in \mathbb{N}^r$ such that $\sum_{i=0}^{r-1} \mathit{k}_i = n$. We denote a zonal spherical function indexed by $(\mathit{k}_0, \dots, \mathit{k}_{r-1})$ evaluated on a double coset parametrised by $(\mathit{l}_0, \dots, \mathit{l}_{r-1})$ as $\omega^{(\mathit{k}_0, \dots, \mathit{k}_{r-1})}_{(\mathit{l}_0, \dots, \mathit{l}_{r-1})}$.

\begin{theorem}[{\cite[Theorem 4.6]{mizukawa04}}]
The zonal spherical functions of $(G(r,1,n), S_n)$ have the $(n+1,m+1)$-hypergeometric expressions:
$$\omega^{(\mathit{k}_0, \dots, \mathit{k}_{r-1})}_{(\mathit{l}_0, \dots, \mathit{l}_{r-1})} = F((-\mathit{l}_1, \dots, -\mathit{l}_{r-1}),(-\mathit{k}_1, \dots, -\mathit{k}_{r-1}), -n,  \widetilde{\Xi}_{r}).$$
Here $\widetilde{\Xi}_{r} = (1-\xi^{ij})_{1 \leq i,j \leq r-1}$. \qed
\label{19}
\end{theorem}

From now on, we will write $\mathit{k} = (\mathit{k}_1, \dots \mathit{k}_{r-1})$ and $\mathit{l} = (\mathit{l}_1, \dots, \mathit{l}_{r-1})$.

\section{Spherical functions for $(G(r,d,n), S_n)$}

We will calculate the spherical functions of the Gelfand pair $(G(r,d,n), S_n)$ using our knowledge of the spherical functions of the Gelfand pair $(G(r,1,n), S_n)$.

\begin{lemma}
Suppose $(G,H)$ is a Gelfand pair, $\omega$ a spherical function of this Gelfand pair and $K \subseteq G$ is a subgroup of $G$ such that $H \subseteq K$ and $(K,H)$ is a Gelfand pair. Then $\omega |_K$ is a spherical function for the Gelfand pair $(K,H)$. All spherical functions of $(K,H)$ are obtained in this way.
\label{restriction}
\end{lemma}

\begin{proof}
Notice that if $\omega \in C(G)$ is right and left $H$-invariant implies $\omega|_K \in C(H \backslash K / H) \subseteq C(K)$ is left and right $H$-invariant and $\omega|_K \neq 0$. By Theorem \ref{15} the function $\omega|_K$ satisfies equation \eqref{eq:sumformula} for all $k \in K$ and so it is zonal spherical.

If $\Omega$ is a spherical function for the Gelfand pair $(K,H)$, $\Omega \in C(K)$, we can extend $\Omega$ to $C(G)$ as a bi-$H$-invariant function, which we will call $\omega$. We define $\omega(g) := \Omega(g)$ for $g \in K$ and $\omega(g) := 0$ for $g \in G \backslash K$. Then:
$$\omega = \sum c_i \omega_i.$$
We restrict again to $K$ (in which case $\omega_i |_K = \Omega_i$ are spherical functions). Then:
$$\Omega = \sum c_i \Omega_i.$$
We know that the spherical functions form a basis for $C(H \backslash K / H)$, thus we conclude that $\Omega_i = \Omega$ for one of the zonal spherical functions.
\end{proof}

The spherical functions of $G(r,d,n)$ come from restriction of the spherical functions of $G(r,1,n)$. Let $p = \frac{r}{d}$ and $\gamma = (0 \ 1 \dots r-1)^{p} \in S_r$. Then $\Gamma = \langle \gamma \rangle \cong C_d$. Let $S_r$ act on the $r$-tuple $(\mathit{k}_0, \dots, \mathit{k}_{r-1})$ in the usual way, i.e. $\sigma (\mathit{k}_0, \dots, \mathit{k}_{r-1}) = (\mathit{k}_{\sigma^{-1}(0)}, \dots, \mathit{k}_{\sigma^{-1}(r-1)})$. We restrict $V^{(\mathit{k}_0, \dots, \mathit{k}_{r-1})}$ to $G(r,d,n)$. These restricted representations do not have to be irreducible in general, but we do know they contain an irreducible representation contained in $1_{S_n}^{G(r,d,n)}$. This irreducible space is generated by $w = \sum_{f \in M_n(\psi(\mathit{k}_0, \dots, \mathit{k}_{r-1}))} f$ over $\mathbb{C}[G(r,d,n)]$, an $\\ S_n$-invariant element in $V^{(\mathit{k}_0, \dots, \mathit{k}_{r-1})}$. By $W^{(\mathit{k}_0, \dots, \mathit{k}_{r-1})}$ we will denote this space generated by $w$. We notice that the spherical functions connected to $W^{(\mathit{k}_0, \dots, \mathit{k}_{r-1})}$ are the restrictions to $G(r,d,n)$ of the spherical functions connected to $V^{(\mathit{k}_0, \dots, \mathit{k}_{r-1})}$.\\

\begin{definition}
Let $\phi_1$ be the linear map:
$$\phi_1: V^{(\mathit{k}_0, \dots, \mathit{k}_{r-1})} \rightarrow V^{\gamma (\mathit{k}_0, \dots, \mathit{k}_{r-1})}$$
$$f(x_1, \dots, x_n) \mapsto x_1^p \dots x_n^p f(x_1, \dots, x_n) \mod R.$$
Here the mod $R$ means that this map is modulo the relationship $x_i^r = 1$ for all $i$. We write $\phi_a = \phi_1^a$.
\label{defmaps}
\end{definition}

\begin{proposition}
The map $\phi_1$ gives an isomorphism of $\mathbb{C}[G(r,d,n)]$-modules:
$$V^{(\mathit{k}_0, \dots, \mathit{k}_{r-1})} \cong V^{\gamma (\mathit{k}_0, \dots, \mathit{k}_{r-1})}.$$
Here $\gamma = (0 \ 1 \dots r-1)^p$.
\label{isom}
\end{proposition}

\begin{proof}
First notice that the map $\phi_1$ is a bijection, because there is an inverse $\phi_1^{-1} = \phi_{d-1}$. We need to show that it is an intertwiner. Let $(\xi^{a_1}, \dots, \xi^{a_n}, \sigma) \in G(r,d,n)$, then $\\ p(a_1 + \dots + a_n) \equiv 0 \ \mathrm{mod} \ r$.
Now:
\begin{flalign*}
& (\xi^{a_1}, \dots, \xi^{a_n}, \sigma) \phi_1(f(x_1, \dots, x_n)) = (\xi^{a_1}, \dots, \xi^{a_n}, \sigma) (x_1^p \dots x_n^p f(x_1, \dots, x_n)) \myeq
\\ & (\xi^{-a_\sigma(1)}x_{\sigma(1)})^p \dots (\xi^{-a_\sigma(n)}x_{\sigma(n)})^p f(\xi^{-a_{\sigma(1)}}x_{\sigma(1)}, \dots, \xi^{-a_{\sigma(n)}}x_{\sigma(n)}) =
\\ & x_1^p \dots x_n^p f(\xi^{-a_{\sigma(1)}}x_{\sigma(1)}, \dots, \xi^{-a_{\sigma(n)}}x_{\sigma(n)}) = \phi_1((\xi^{a_1}, \dots, \xi^{a_n}, \sigma) f(x_1, \dots, x_n)).
\end{flalign*}
Notice that (1) still holds with the extra identity $x_i^r = 1$ because $(\xi^{a_i})^r =1$ holds. Thus the map $\phi$ is $\mathbb{C}[G(r,d,n)]$-linear and so the two modules are isomorphic as $\\ \mathbb{C}[G(r,d,n)]$-modules.
\end{proof}

\begin{corollary}
Two zonal spherical functions of $(G(r,1,n), S_n)$ restricted to $G(r,d,n)$ indexed by $(\mathit{k}_0, \dots, \mathit{k}_{r-1})$ and $(\mathit{k}'_0, \dots, \mathit{k}'_{r-1})$ are the same if there is $\beta \in \Gamma$ such that $\\ \beta (\mathit{k}_0, \dots, \mathit{k}_{r-1}) = (\mathit{k}'_0, \dots, \mathit{k}'_{r-1})$.
\end{corollary}

\begin{proof}
This follows from the isomorphism of Proposition \ref{isom}.
\end{proof}

In Corollary \ref{indexcor}, we will see that the restrictions of the zonal spherical functions are only the same when connected to $V^{\gamma^k (\mathit{k}_0, \dots, \mathit{k}_{r-1})}$ for any $k$. Now let $G = G(r,1,n)$, $K = G(r,d,n)$, $H = S_n$ and $p = \frac{r}{d}$. Now $G / K \cong C_d = \langle \xi^p \rangle$. Identify $G / K \cong C_d$ with the subgroup $\{ (\mu, \dots, \mu, id) \in G(r,1,n) \ | \ \mu \in C_d \}$. We notice that the group $H$ commutes with $\{ (\mu, \dots, \mu, id) \in G(r,1,n) \ | \ \mu \in C_d \}$. We will denote by $(G/K)H$ the product of the group $\{ (\mu, \dots, \mu, id) \in G(r,1,n) \ | \ \mu \in C_d \}$ with $H$, which is isomorphic to $G/K \times H$. By Proposition \ref{16} each double coset in $H \backslash G / H$ can be identified with an $r$-tuple $(\mathit{l}_0, \dots, \mathit{l}_{r-1})$ such that $\sum_{i=0}^{r-1} \mathit{l}_i = n$. We see that in the case of $H \backslash G / ((G/K)H)$ the double cosets can be identified with the orbits of $(\mathit{l}_0, \dots, \mathit{l}_{r-1})$ under the group $\Gamma = \langle (0\dots r-1)^p \rangle \subseteq S_r$. We write $X_n^{r-1} = \{(\mathit{k}_0, \dots, \mathit{k}_{r-1}) | \sum_{i=0}^{r-1} \mathit{k}_i = n\}$. If $|H \backslash K / H| = |H \backslash G / ((G / K)H)|$ then the spherical functions are exactly indexed by the orbits in  $X_n^{r-1} / \Gamma$. This is because on the orbits of $\Gamma \cong C_d$ the zonal spherical functions are the same, so there can be at most $|H \backslash G / ((G / K)H)|$ different restrictions and restriction is surjective, so there are $|H \backslash K / H|$ different restrictions (the amount of spherical functions).

\begin{lemma}
Let $r$ be an integer and $(l_1, \dots, l_n) \in \mathbb{N}^n$. The number of $n$-tuples \\ $(a_1, \dots, a_n) \in (\mathbb{Z}/r\mathbb{Z})^n$ satisfying $a_1 l_1 + \dots + a_n l_n \ \mathrm{mod} \ r = 0 \ \mathrm{mod} \ r$ is $\mathrm{gcd}(r, l_1, \dots, l_n) r^{n - 1}$
\label{K_lemma}
\end{lemma}

\begin{proof}
Set $k = \mathrm{gcd}(r, l_1, \dots, l_{n})$. First, let $r$ be a prime power, i.e. $r = p^a$. We know that $k = \mathrm{gcd}(r, l_1, \dots, l_{n}) =  p^b$ for some $0 \leq b \leq a$. We assume without loss of generality that $\mathrm{gcd}(r, l_1) = p^b$, which means that $\langle l_1\ \mathrm{mod} \ r \rangle = \langle p^b \ \mathrm{mod} \ r \rangle$ as subgroups of $\mathbb{Z}/r\mathbb{Z}$. We now choose $a_i$ for $i > 1$ freely in $r^{n-1}$ ways, and we wonder how many $a_1$ there can be such that:
\begin{flalign*}
& a_1 l_1 + \dots + a_{n} l_{n} \ \mathrm{mod} \ r = 0 \ \mathrm{mod} \ r
\\ & a_1 l_1 \ \mathrm{mod} \ r = (- a_2 l_2 - \dots - a_{n} l_{n}) \ \mathrm{mod} \ r.
\end{flalign*}
At least one such $a_1$ exists, so this questions is synonymous with asking how many $c$ there are such that $c l_1 \ \mathrm{mod} \ r = 0 \ \mathrm{mod} \ r$. This number is $\mathrm{gcd}(l_1,r) = p^b = k$. Hence in this case the amount of $n$-tuples $(a_1, \dots, a_n) \in (\mathbb{Z}/r\mathbb{Z})^n$ satisfying our conditions is $kr^{n - 1}$.

Now let $r$ be a general integer, then $r = p_1^{b_1} \dots p_m^{b_m}$ with the $p_i$ all distinct prime numbers. We notice that:
$$k = \mathrm{gcd}(r, l_1, \dots, l_{n}) = \prod_{i=1}^m \mathrm{gcd}(p^{b_i}, l_1, \dots, l_{n}).$$
We study the $(a_1, \dots, a_{n}) \in (\mathbb{Z}/r\mathbb{Z})^n$ such that $l_1 a _1 + \dots + l_{n} a_{n} \ \mathrm{mod} \ r = 0 \ \mathrm{mod} \ r$. By the Chinese Remainder Theorem we know that this coincides with the $\\ ((a_1^1, \dots, a_1^m), \dots, (a_{n}^1, \dots, a_{n}^m)) \in (\mathbb{Z} / p^{b_1} \mathbb{Z} \times \dots \times \mathbb{Z}/ p^{b_m} \mathbb{Z})^{n}$ such that (where in $a^i$ $i$ is now used as an index):
$$\forall_i \ l_1 a_1^i + \dots + l_{n} a_{n}^i \ \mathrm{mod} \ p^{b_i} = 0 \ \mathrm{mod} \ p^{b_i}.$$
By the earlier part of the proof, there are $\mathrm{gcd}(p^{b_i}, l_1, \dots, l_{n})(p^{b_i})^{n-1}$ such $(a_1^i, \dots, a_{n}^i)$. This means the amount of $n$-tuples $(a_1, \dots, a_{n}) \in (\mathbb{Z}/r\mathbb{Z})^n$ such that $l_1 a _1 + \dots + l_{n} a_{n} \ \mathrm{mod} \ r = 0 \ \mathrm{mod} \ r$ is $\prod_{i=1}^m (\mathrm{gcd}(p^{b_i}, l_1, \dots, l_{n})(p^{b_i})^{n-1}) = \mathrm{gcd}(r, l_1, \dots, l_{n}) r^{n - 1} = k  r^{n - 1}$.
\end{proof}

\begin{proposition}
$|H \backslash K / H| = |H \backslash G / ((G/K)H)|$.
\end{proposition}

\begin{proof}
First we notice that each left coset of $K/H$ can be represented by a $(\xi^{a_1}, \dots, \xi^{a_n})$ such that $a_1 + \dots + a_n \ \mathrm{mod} \ d = 0 \ \mathrm{mod} \ d$ and each left coset of $G / ((G/K)H)$ can be represented by a $(\xi^j,\xi_2, \dots, \xi_n)$ where $\xi_i \in C_r$ and $0 \leq j \leq p-1$. Each such element uniquely determines a coset. We notice that $H$ acts from the left on $X = G / ((G/K)H)$ and $Y = K/H$ and that the orbits under these actions can be identified with the double cosets of $H \backslash K / H$ and $H \backslash G / ((G/K)H)$. Hence by Burnside's Lemma:
$$|H \backslash G / ((G/K)H)| = \frac{1}{|H|} \sum_{\sigma \in H} |X^{\sigma}|$$
$$|H \backslash K / H| = \frac{1}{|H|} \sum_{\sigma \in H} |Y^{\sigma}|.$$
Remember $H = S_n$. We will prove $\forall_{\sigma \in S_n} \ |X^{\sigma}| = |Y^{\sigma}|$. If $\sigma \in S_n$, then $\sigma$ can be written as a product of disjoint cycles, i.e. $\sigma = \tau_1 \dots \tau_{n_{\sigma}}$. Let $l_i$ be the length of the cycle $\tau_i$ and $n_{\sigma}$ the number of cycles. We write $k = \mathrm{gcd}(d, l_1, \dots, l_{n_{\sigma}})$.

\textit{Claim 1:} $|Y^{\sigma}| = kpr^{n_{\sigma}-1}$.

We know that for every $y \in Y$ there is a representative $(\xi^{a_1}, \dots, \xi^{a_n})$ such that $\\ a_1 + \dots + a_n \ \mathrm{mod} \ d = 0 \ \mathrm{mod} \ d$. For $(\xi^{a_1}, \dots, \xi^{a_n}) \in Y^{\sigma}$ to hold we know that $a_i$ must be constant for each $i$ in the same cycle $\tau_j$. So we need to count the $n_{\sigma}$ - tuples $\\ (a_1, \dots, a_{n_\sigma}) \in (\mathbb{Z}/ r \mathbb{Z})^{n_{\sigma}}$ such that $a_1l_1 + \dots + a_{n_{\sigma}} l_{n_{\sigma}} \ \mathrm{mod} \ d = 0 \ \mathrm{mod} \ d$. By Lemma \ref{K_lemma}  there are $k d^{n_{\sigma} - 1}$ such integers where the $a_i \in \mathbb{Z} / d \mathbb{Z}$. However, the original $a_i$ can be in $\mathbb{Z} / r \mathbb{Z}$. This means that the amount of $n_{\sigma}$-tuples $(a_1, \dots, a_{n_\sigma}) \in (\mathbb{Z}/ r \mathbb{Z})^{n_{\sigma}}$ such that $a_1l_1 + \dots + a_{n_{\sigma}} l_{n_{\sigma}} \ \mathrm{mod} \ d = 0 \ \mathrm{mod} \ d$ is $k d^{n_{\sigma} - 1}p^{n_{\sigma}} = \mathrm{gcd}(d, l_1, \dots, l_{n_{\sigma}}) pr^{n_{\sigma} - 1}$. Hence $|Y^{\sigma}| = kpr^{n_{\sigma}-1}$.

\textit{Claim 2:} $|X^{\sigma}| = kpr^{n_{\sigma}-1}$.

 If $x \in X$, choose a representative $(\xi^j, \xi_2, \dots, \xi_n)$, where $0 \leq j \leq p-1$. Now let $\sigma$ act on $(\xi^j, \xi_2, \dots, \xi_n)$ and let $\xi_1 = \xi^j$. Then $\sigma (\xi_1, \xi_2, \dots, \xi_n) = (\xi_{\sigma^{-1}(1)}, \dots, \xi_{\sigma^{-1}(n)})$ where $(\xi_{\sigma^{-1}(1)}, \dots, \xi_{\sigma^{-1}(n)})$ represents the class $(\xi_{\sigma^{-1}(1)}, \dots, \xi_{\sigma^{-1}(n)}, id)(G/K)H$. We know that $\sigma = \tau_1 \dots \tau_{n_{\sigma}}$ where the $\tau_i$ are disjoint cycles of length $l_i$. We know that $\sigma$ acts trivially if and only if there is $\mu \in C_d$ such that: $\forall_i \ \xi_i = \mu \xi_{\sigma^{-1}(i)}$. We obtain $\mu^{l_i} = 1$ for every $l_i$. On the other hand, if we have such a $\mu$, it can give $p r^{n_{\sigma} - 1}$ distinct $(\xi^j, \xi_2, \dots, \xi_n) \in X^{\sigma}$ (decide $\xi_i$ for exactly one number in each cycle $\tau_j$). We obtain:
$$|X^{\sigma}| = |\{ \mu \in C_d \ | \ \forall_i \ \mu^{l_i} = 1 \}|pr^{n_{\sigma} - 1}.$$
So we need to count the number of $\mu$ satisfying our conditions. Recall that $\mu = \xi^{ap}$ for an $a \in \mathbb{Z}/d \mathbb{Z}$. Now we need to count the $a \in \mathbb{Z}/d\mathbb{Z}$ such that $\forall_i  \ a \cdot l_i \ \mathrm{mod} \ d = 0 \ \mathrm{mod} \ d$. We know that by the Euclidean algorithm there are integers $m, x_i$ such that:
\begin{flalign*}
& x_1 l_1 + \dots + x_{n_{\sigma}} l_{n_{\sigma}} \ \mathrm{mod} \ d= k \ \mathrm{mod} \ d.
\end{flalign*}
If we multiply by an $a$ satisfying our conditions we get:
$$0 \ \mathrm{mod} \ d= a (x_1 l_1 + \dots + x_{n_{\sigma}} l_{n_{\sigma}}) \ \mathrm{mod} \ d = ak \ \mathrm{mod} \ d.$$
Hence the number of $a$ satisfying our conditions is $|\langle \frac{d}{k} \rangle| = k$. So $|X^{\sigma}| = kpr^{n_{\sigma} - 1} = |Y^{\sigma}|$.
\end{proof}

\begin{corollary}
The zonal spherical functions of $G(r,d,n)$ are exactly indexed by the orbits of $X_n^{r-1}$ under the group $\Gamma$. \qed
\label{indexcor}
\end{corollary}

Now we know the spherical functions are indexed by the orbits of $X_n^{r-1}$ under the group $\Gamma$, we will also be wanting to know about the dimensions of the irreducible representations $W^{(\mathit{k}_0, \dots, \mathit{k}_{r-1})}$ belonging to the zonal spherical functions so we can deduce the values of the inner product $\frac{1}{|G|} \sum_{k=1}^s |D_k| \omega_i(D_k)\overline{\omega_j(D_k)} = \delta_{ij} \mathrm{dim} (V_i)^{-1}$ from Corollary \ref{13}. Denote the stabilizer of an element $x \in X_n^{r-1}$ by $\Gamma_x$ and the orbit of $x$ by $\Gamma x$. We know that if $S \subseteq X_n^{r-1}$ is a set of representatives of the orbits of $X_n^{r-1}$ by $\Gamma$ then $1_{S_n}^{G(r,d,n)} = \bigoplus_{s \in S} W^s$.

\begin{proposition}
$\dim(W^{(\mathit{k}_0, \dots, \mathit{k}_{r-1})}) = \frac{\dim(V^{(\mathit{k}_0, \dots, \mathit{k}_{r-1})})}{|\Gamma_{(\mathit{k}_0, \dots, \mathit{k}_{r-1})}|} = |\Gamma_{(\mathit{k}_0, \dots, \mathit{k}_{r-1})}|^{-1} \binom{n}{\mathit{k}_0, \dots, \mathit{k}_{r-1}}$.
\end{proposition}

\begin{proof}
Remember that for an arbitrary finite group acting on a finite space $X$ the equality $|G| = |G_x| \cdot |Gx|$ holds. We will prove that $\dim(W^{(\mathit{k}_0, \dots, \mathit{k}_{r-1})}) \leq \frac{\dim(V^{(\mathit{k}_0, \dots, \mathit{k}_{r-1})})}{|\Gamma_{(\mathit{k}_0, \dots, \mathit{k}_{r-1})}|}$. Once we know that, then:
\begin{flalign*}
& \frac{r^n}{d} = \dim(1_{S_n}^{G(r,d,n)}) = \sum_{s \in S} \dim(W^s) \leq \sum_{s \in S} \frac{\dim(V^{s})}{|\Gamma_{s}|} = \sum_{x \in X_n^{r-1}} \frac{\dim(V^x)}{|\Gamma_x| |\Gamma x|} =
\\ & \frac{1}{|\Gamma|} \sum_{x \in X_n^{r-1}} \dim(V^x) = \frac{1}{d} r^n = \frac{r^n}{d}.
\end{flalign*}
The above inequality is an equality and we would obtain that:
$$\dim(W^{(\mathit{k}_0, \dots, \mathit{k}_{r-1})}) = \\ \frac{\dim(V^{(\mathit{k}_0, \dots, \mathit{k}_{r-1})})}{|\Gamma_{(\mathit{k}_0, \dots, \mathit{k}_{r-1})}|}.$$

If $g \in G(r,1,n)$, then $g W^{(\mathit{k}_0, \dots, \mathit{k}_{r-1})}$ is a $\mathbb{C}[G(r,d,n)]$ - submodule because $G(r,d,n)$ is normal. By the irreducibility of $W^{(\mathit{k}_0, \dots, \mathit{k}_{r-1})}$ we obtain that each $gW^{(\mathit{k}_0, \dots, \mathit{k}_{r-1})}$ is irreducible and either $g_1 W^{(\mathit{k}_0, \dots, \mathit{k}_{r-1})} \cap g_2 W^{(\mathit{k}_0, \dots, \mathit{k}_{r-1})} = \{0\}$ or $g_1 W^{(\mathit{k}_0, \dots, \mathit{k}_{r-1})} = g_2 W^{(\mathit{k}_0, \dots, \mathit{k}_{r-1})}$. Hence $V^{(\mathit{k}_0, \dots, \mathit{k}_{r-1})}$ is a direct sum of some of these spaces and $\dim(W^{(\mathit{k}_0, \dots, \mathit{k}_{r-1})}) | \dim(V^{(\mathit{k}_0, \dots, \mathit{k}_{r-1})})$.

We know that $\Gamma_{(\mathit{k}_0, \dots, \mathit{k}_{r-1})} \subseteq \Gamma$ is a subgroup and because $\Gamma$ is cyclic we know that $\Gamma_{(\mathit{k}_0, \dots, \mathit{k}_{r-1})} = \langle (0 \dots r-1)^{ap} \rangle$ for some $a | d$ and $|\Gamma_{(\mathit{k}_0, \dots, \mathit{k}_{r-1})}| = \frac{d}{a}$. If $\gamma \in \Gamma_{(\mathit{k}_0, \dots, \mathit{k}_{r-1})}$ the function $\phi_a$ from Definition \ref{defmaps} is a $\mathbb{C}[G(r,d,n)]$-endomorphism of $V^{(\mathit{k}_0, \dots, \mathit{k}_{r-1})}$. We study what this function does on the spaces $gW^{(\mathit{k}_0, \dots, \mathit{k}_{r-1})}$. First notice that if $w = \sum_{f \in M_n(\psi(\mathit{k}_0, \dots, \mathit{k}_{r-1}))} f$ then $\phi_a(w) = w$. Now let $g = (\xi_1, \dots, \xi_n, \sigma) \in G(r,1,n)$. If $v \in gW^{(\mathit{k}_0, \dots, \mathit{k}_{r-1})}$, i.e. $v = g c w$, where $c \in \mathbb{C}[G(r,d,n)]$, we see that:
\begin{flalign*}
& \phi(gcw) =  x_1^{ap} \dots x_n^{ap} (gcw) = (\xi_1 \dots \xi_n)^{ap} (g (x_1 \dots x_n)^{ap}) (gcw)
\\ & (\xi_1 \dots \xi_n)^{ap} g \phi_a(cw) = (\xi_1 \dots \xi_n)^{ap} g c \phi_a(w) = (\xi_1 \dots \xi_n)^{ap} gcw.
\end{flalign*}

Hence the $g W^{(\mathit{k}_0, \dots, \mathit{k}_{r-1})}$ are eigenspaces of $\phi$ with eigenvalues $(\xi_1 \dots \xi_n)^{ap}$. We see there can be $\frac{d}{a}$ different eigenvalues and thus at least $\frac{d}{a} = |\Gamma_{(\mathit{k}_0, \dots, \mathit{k}_{r-1})}|$ different eigenspaces. We know that these eigenspaces must be disjoint and each of these eigenspaces contains $gW^{(\mathit{k}_0, \dots, \mathit{k}_{r-1})}$ for some $g$. This means $\frac{d}{a}\dim(W^{(\mathit{k}_0, \dots, \mathit{k}_{r-1})}) \leq \dim(V^{(\mathit{k}_0, \dots, \mathit{k}_{r-1})})$ and thus:
$$\dim(W^{(\mathit{k}_0, \dots, \mathit{k}_{r-1})}) \leq \frac{\dim(V^{(\mathit{k}_0, \dots, \mathit{k}_{r-1})})}{\Gamma_{(\mathit{k}_0, \dots, \mathit{k}_{r-1})}}.$$
\end{proof}

Let $Y_{n,d}^{r-1} = \{(\mathit{l}_1, \dots, \mathit{l}_{r-1}) | \sum \mathit{l}_i \leq n, \sum i \mathit{l}_i = 0 \ \mathrm{mod} \ d\}$ where $d | r$. We get some orthogonality relations for the hypergeometric functions.

\begin{corollary}
If $(\mathit{k}_0, \dots, \mathit{k}_{r-1}), (\mathit{k}_0', \dots, \mathit{k}_{r-1}') \in X_n^{r-1}$ and we set $\mathit{l}_0 = n - \mathit{l}_1 - \dots - \mathit{l}_{r-1}$:
\begin{flalign*}
& \frac{d}{r^n} \sum_{\mathit{l} \in Y_{n,d}^{r-1}} \binom{n}{\mathit{l}_0, \dots, \mathit{l}_{r-1}}F(-\mathit{l}, -\mathit{k}, -n, \widetilde{\Xi}_{r})\overline{F(-\mathit{l}, -\mathit{k}', -n, \widetilde{\Xi}_{r})} =
\\ & 1_{\Gamma (\mathit{k}_0, \dots, \mathit{k}_{r-1})} ((\mathit{k}_0', \dots, \mathit{k}_{r-1}')) |\Gamma_{(\mathit{k}_0, \dots, \mathit{k}_{r-1})}| \binom{n}{\mathit{k}_0, \dots, \mathit{k}_{r-1}}^{-1}.
\end{flalign*}
Here $1_X$ is the indicator function.
\end{corollary}

\begin{proof}
This follows from the orthogonality relations of spherical functions in Corollary \ref{13} applied to the spherical functions of the Gelfand pair $(G(r,d,n), S_n)$.
\end{proof}

We will now give some examples of Gelfand pairs of complex reflection groups for illustration and to look at some notable groups.

\vspace{.5 cm}

\textbf{Example of the dihedral group}

We will now look at the special case $G(r,r,2)$, the dihedral group $D_r$. This group is generated by the elements $a$ and $b$ with the relations $a^r = b^2 = (ab)^2 = e$ the identity. We will show that the spherical functions coincide with the known spherical functions of $D_r$ in the literature. Note that $G(r,r,2) = \{ (\xi^{i}, \xi^{-i}, (12)^j) \ | \ i, j \in \mathbb{Z}\}$. This coincides with the dihedral group if we set $a \sim (\xi, \xi^{-1}, id)$ and $b \sim (1, 1, (12))$. Let $D_1 = \langle b \rangle = \{e,b\} \cong S_2$. We know $(D_r, D_1)$ or $(G(p,p,2), S_2)$ is a Gelfand pair.

\begin{proposition}[{\cite[Theorem 4.3]{akamizu}}]
Each double coset can be represented by a $k$ with $0 \leq k \leq \lfloor \frac{r}{2} \rfloor$, the double cosets being represented by $D_1 a^k D_1 = \{a^k, a^k b, a^{-k}, a^{-k}b \}$. The spherical functions of $(D_r, D_1)$ are indexed by $0 \leq m \leq \lfloor \frac{r}{2} \rfloor$ and are given by:
$$\omega_m(a^{k}) = \cos(\frac{2 \pi k m}{r}).$$ \qed
\end{proposition}

The spherical functions of $(G(p,p,2),S_2)$ we found are indexed by the orbits of $\\ X_2^{r-1} = \{(\mathit{k}_0, \dots, \mathit{k}_{r-1}) | \sum_{i=0}^{r-1} \mathit{k}_i = 2\}$ under the group $\Gamma = \langle (0 \dots r-1) \rangle \subseteq S_r$. A full set of representatives is $\{e_0 + e_i\}_{0 \leq i \leq \lfloor \frac{r}{2} \rfloor}$ (here $e_i = (\delta_{ki})_{0 \leq k \leq r-1}$). We then see that the double cosets corresponding to $D_1 a^k D_1$ ($0 \leq k \leq \lfloor \frac{r}{2} \rfloor$) are indexed by $(e_{k} + e_{- k \mod r})$. We find that on the orbit represented by $e_0 + e_i$, where $0 \leq i \leq \lfloor \frac{r}{2} \rfloor$:
$$\omega^{e_0 + e_i}_{e_k + e_{- k \mod r}} =1 - \frac{1}{2}(1 - \xi^{ik}) - \frac{1}{2}(1 - \xi^{-ik}) = \cos(\frac{2 \pi ik}{r}).$$
Hence the spherical functions here and those resulting from \cite{akamizu} are the same, as they should be.

\vspace{.5 cm}

\textbf{Example of $G(2,2,n)$}

The group $G(2,2,n) \subseteq G(2,1,n)$ is a subgroup of the hyperoctahedral group $G(2,1,n) = H_n$ and is the Weyl group of type $D_n$ \cite{humphreys}. We will further study the spherical functions of this group. We know that:
$$G(2,2,n) = \{((-1)^{a_1}, \dots, (-1)^{a_n}, \sigma) \in G(2,1,n) \ | \ a_1 + \dots + a_n \mod 2 \equiv 0 \}.$$
By Proposition \ref{16} the double cosets in $S_n \backslash G(2,2,n) / S_n$ can be represented by a $2$-tuple $(\mathit{l}_0, \mathit{l}_1)$ such that $\mathit{l}_0 + \mathit{l}_1 = n$ and $\mathit{l}_1 \mod 2 \equiv 0$, i.e. $\mathit{l}_1$ is even. This means each double coset is uniquely decided by an even number $\mathit{l}_1$ such that $0 \leq \mathit{l}_1 \leq n$.

\begin{proposition}
The spherical functions of $G(2,2,n)$ are Gauss hypergeometric functions indexed by the $0 \leq \mathit{k} \leq \lfloor \frac{n}{2} \rfloor$, where the value on the double coset indexed by $\mathit{l}_1$ on the spherical function indexed by $\mathit{k}$ is:
$$\omega_{(\mathit{l}_0, \mathit{l}_1)}^{n-\mathit{k}, \mathit{k}} = {}_2 F_1 (-\mathit{l}_1, -\mathit{k}, -n, 2).$$
\end{proposition}

\begin{proof}
The spherical functions are indexed by the orbits of the $2$-tuples $(\mathit{k}_0, \mathit{k}_1)$ under the group $\Gamma = \langle (12) \rangle = S_2$. We take the representative of each orbit such that $\mathit{k}_0 \geq \mathit{k}_1$, which means $0 \leq \mathit{k}_1 \leq \lfloor \frac{n}{2} \rfloor$. Each such $\mathit{k}_1$ gives a unique orbit.
\end{proof}

\section{Product formula}

Recall the product formula of Theorem \ref{15}. In this section we will apply it to explicitly find a product formula for hypergeometric functions. We consider the Gelfand pair $(G(r,1,n),S_n)$.

\begin{theorem}
For the $(n+1,m+1)$-hypergeometric functions we have a product formula:
\begin{flalign*}
& F(-\mathit{l}, -\mathit{k}, -n, \widetilde{\Xi}_{r}) \cdot F(-\mathit{l}', -\mathit{k}, -n, \widetilde{\Xi}_{r}) =
\\ & \binom{n}{\mathit{l}_0', \dots, \mathit{l}_{r-1}'}^{-1} \cdot \sum_{A \in M_{r-1 \times r-1}(\mathbb{N})}(\prod_{i=0}^{r-1}\binom{\mathit{l}_i}{a_{i0}, \dots, a_{i(r-1)}}) \cdot
\\ & F((-\sum_{i=0}^{r-1}a_{i((-i + 1) \mathrm{mod} r)}, \dots, -\sum_{i=0}^{r-1}a_{i((r-1-i) \mathrm{mod} r)}),-\mathit{k}, -n, \widetilde{\Xi}_{r})
\end{flalign*}
where $\mathit{l}$, $\mathit{l}'$ and $\mathit{k}$ are $(r-1)$-tuples in $\mathbb{N}^{r-1}$, $n \in \mathbb{N}$, $\sum_{i=1}^{n}\mathit{l}_i \leq n, \ \sum_{i=1}^{n}\mathit{l}_i' \leq n$, $\sum_{i=1}^{n}\mathit{k}_i \leq n$, $\widetilde{\Xi}_{r} = (1-\xi^{ij})_{1 \leq i,j \leq r-1}$ and $\xi = \exp((2\pi i)/r)$, $\mathit{l}_0 = n - \sum_{i=1}^{n}\mathit{l}_i$, $\mathit{l}_0' = n - \sum_{i=1}^{n}\mathit{l}_i'$, $\\ a_{0i} = l_{i}' - \sum_{j=1}^{r-1}a_{ji}$ $(i > 0)$, $a_{i0} = l_{i} - \sum_{j=1}^{r-1}a_{ij}$ $(i > 0)$, $a_{00} = n + \sum_{i,j \geq 1}a_{ij} - \sum_{i=1}^{n}\mathit{l}_i - \sum_{i=1}^{n}\mathit{l}_i'$.
\end{theorem}

\begin{proof}
Recall Theorem \ref{15}. By Theorem \ref{19}, for $g = (g_1, \dots, g_n, \sigma )$ and $|\{j \ | \  g_j = \xi ^ i\}| = \mathit{l}_i$ the zonal spherical function associated to $(\mathit{k}_0, \dots, \mathit{k}_n)$ evaluated in $g$ is:
$$F((-\mathit{l}_1, ... ,-\mathit{l}_{r-1}), (-\mathit{k}_1, ... ,-\mathit{k}_{r-1}), -n, \widetilde{\Xi}_{r}).$$
Here $F$ is an $(n+1,m+1)$-hypergeometric function. We set $\\ x = (\underbrace{1, \dots, 1}_{\mathit{l}_0}, \underbrace{\xi, \dots, \xi}_{\mathit{l}_1}, \dots, \underbrace{\xi ^{r-1}, \dots, \xi ^{r-1}}_{\mathit{l}_{r-1}}, id)$ and $\\ y = (\underbrace{1, \dots, 1}_{\mathit{l}'_0}, \underbrace{\xi, \dots, \xi}_{\mathit{l}'_1}, \dots, \underbrace{\xi ^{r-1}, \dots, \xi ^{r-1}}_{\mathit{l}'_{r-1}}, id)$. Let $\omega$ is the spherical function indexed by $(\mathit{k}_0, \dots, \mathit{k}_{r-1})$. We know by equation \eqref{eq:sumformula} that $\omega(x)\omega(y) = \frac{1}{n!} \sum_{\sigma \in S_n} \omega(x \sigma y)$, which is equal to the left-hand side of the formula in the theorem. We study $x \sigma y$:

$$x \sigma y = (x_1 y_{\sigma ^{-1} (1)}, \dots, x_n y_{\sigma ^{-1} (n)}, \sigma) \sim (x_1 y_{\sigma ^{-1} (1)}, \dots, x_n y_{\sigma ^{-1} (n)}, 1).$$

This means by equation \eqref{eq:sumformula}: 

$$\omega (x) \omega (y) = \frac{1}{n!} \sum_{\sigma \in S_n} \omega (x_1 y_{\sigma ^{-1} (1)}, \dots, x_n y_{\sigma ^{-1} (n)}, 1) = \frac{1}{n!} \sum_{\sigma \in S_n} \omega (x_1 y_{\sigma (1)}, \dots, x_n y_{\sigma (n)}, id).$$

We have $\sigma y = \tau y$ if and only if $\tau \sigma^{-1} \in S_{\mathit{l}_0'} \times \dots \times S_{\mathit{l}_{r-1}'}$. This means:
\begin{equation}
\omega (x) \omega (y) = \frac{\mathit{l}_0'! \dots \mathit{l}_{r-1}'!}{n!} \sum_{\sigma \in S_n /  (S_{\mathit{l}_0'} \times \dots \times S_{\mathit{l}_{r-1}'})} \omega (x_1 y_{\sigma (1)}, \dots, x_n y_{\sigma (n)}, id).
\label{eq:2}
\end{equation}
Now we create a matrix $A \in M_{r \times r}(\mathbb{N})$, $A = (A_{ij})_{0 \leq i,j \leq r-1}$ for a $\sigma \in S_n /  (S_{\mathit{l}_0'} \times \dots \times S_{\mathit{l}_{r-1}'})$. It is created like this:
$$a_{ij} = | \{ k | x_k = \xi ^ i \} \cap \{ k | y_{\sigma (k)} = \xi ^ j \}|.$$
Then $\sum_{j = 0} a_{ij} = \mathit{l}_i, \sum_{j=0} a_{ji} = \mathit{l}_i'$. If $A$ is a matrix with $\sum_{j = 0} a_{ij} = \mathit{l}_i$ and $\sum_{j=0} a_{ji} = \mathit{l}_i'$ then it can be seen that there is a $\sigma$ that gives rise to this matrix. If some $\sigma$ gives rise to such a matrix, and we associate an $r$-tuple $(\mathit{l}_0^{\sigma}, \dots, \mathit{l}_{r-1}^{\sigma})$ to $(x_1 y_{\sigma (1)}, \dots, x_n y_{\sigma (n)}, 1)$ in the usual way we find that:
$$\mathit{l}_k^{\sigma} = \sum_{\{(i,j)|i+j = k \ mod \ r\}} a_{ij} = \sum_{i=0}^{r-1} a_{i((k-i)mod) \ r}.$$
This means $A$ uniquely decides the double coset.
We see there are $\prod_{i=0}^{r-1} \binom{\mathit{l}_i}{a_{i0}, \dots, a_{i(r-1)}}$ different $\sigma \in S_n /  (S_{\mathit{l}_0'} \times \dots \times S_{\mathit{l}_{r-1}'})$ such that they give rise to the matrix $A$. Now define:
$$\mathcal{A} = \{ A \in M_{r \times r}(\mathbb{N}) \ | \ \sum_{i=0}^{r-1} a_{ji} = \mathit{l}_j, \ \sum_{i=0}^{r-1} a_{ij} = \mathit{l}_j' \}.$$
By \eqref{eq:2} we obtain: 
\begin{flalign*}
& \omega (x) \omega (y) = \binom{n}{\mathit{l}_0', \dots, \mathit{l}_{r-1}'}^{-1} \sum_{A \in \mathcal{A}} (\prod_{i=0}^{r-1} \binom{\mathit{l}_i}{a_{i0}, \dots, a_{i(r-1)}}) \cdot
\\ & F((- \sum a_{i((1-i)mod r)}, \dots, - \sum a_{i((r-1-i)modr)}), -\mathit{k}, -n, \widetilde{\Xi}_r).
\end{flalign*}
Now we relabel the $A \in \mathcal{A}$. If $A \in M_{(r-1) \times (r-1)} (\mathbb{N})$ (where $A=(A_{ij})_{1 \leq i,j \leq r-1}$). Define:
$\\a_{0i} = \mathit{l}_i' - \sum_{j=1}^{r-1}a_{ji} \ (i > 0)$, $a_{i0} = \mathit{l}_i - \sum_{j=1}^{r-1}a_{ij} \ (i > 0)$, $a_{00} = n + \sum_{i,j \geq 1} a_{ij} - \sum_{i \geq 1} \mathit{l}_i' + \mathit{l}_i$.
\end{proof}

\begin{remark}
The spherical functions for $(G(r,d,n), S_n)$ are the same hypergeometric functions as for $(G(r,1,n), S_n)$. This means the product formula gives roughly the same equalities, with only the second argument possibly differing in the hypergeometric function.
\end{remark}

We will prove that the identity, in the case of $r=2$, coincides with a special case of an identity stated by Dunkl \cite{dunkladd}. We use the formulation by Rahman \cite{rahman79}:
\begin{flalign*}
& K_{n}(x; p, N) K_{n}(y; p, N) = \sum_{s=0}^x \binom{x}{s}\frac{(y-N)_s(-y)_{x-s}}{(-N)_x} \sum_{r=0}^{x-s}\binom{x-s}{r} \cdot
\\ & (\frac{2p-1}{p})^r(\frac{1-p}{p})^{x-s-r}K_n(2s+y-x+r;p,N).
\end{flalign*}
Here $K_n(x;p,N) = {}_2F_1(-x,-n;-N, \frac{1}{p})$ the Krawtchouk polynomial. In our case $p = \frac{1}{2}$, so the formula degenerates to:
$$K_{n}(x; \frac{1}{2}, N) K_{n}(y; \frac{1}{2}, N) = \sum_{s = 0}^{x} \binom{x}{s} \frac{(y-N)_{s}(-y)_{x-s}}{(-N)_{x}} K_n(2s +y -x; \frac{1}{2}, N).$$
Our product formula gives us (take $x \leq y$):
\begin{flalign*}
& K_n(x; \frac{1}{2}, N) K_n(y; \frac{1}{2}, N) = K_n(y; \frac{1}{2}, N) K_n(x; \frac{1}{2}, N)
\\ & = \binom{N}{x}^{-1} \sum_{i \geq 0} \binom{N-y}{x-i} \binom{y}{i} K_n(x+y-2i; \frac{1}{2}, N).
\end{flalign*}
Now $\max{\{ 0, x+y-N \}} \leq i \leq x$. Define $s = x-i$, this implies $0 \leq s \leq \min{\{ x, N-y \}}$. Then the sum above equals:
\begin{flalign*}
& = \sum_{s = 0}^{x} \binom{N}{x}^{-1} \binom{N-y}{s} \binom{y}{x-s} K_n(y + 2s - x; \frac{1}{2}, N).
\end{flalign*}
Because $\binom{N}{x}^{-1} \binom{N-y}{s} \binom{y}{x-s} = \binom{x}{s} \frac{(y-N)_s (-y)_{x-s}}{(-N)_x}$ we see that the identities coincide.

\section{Laplace operator}
In this section we will find an analog of the Laplace operator for a finite space of which the spherical functions are eigenfunctions. The motivation for finding such an operator is that for Gelfand pairs of Lie groups there exist such operators and for a certain category of discrete Gelfand pairs $(G,H)$ discussed in \cite[Chapter 5]{cecsil08}.

\begin{definition}
Let the Hamming distance on $G(r,d,n)/S_n$ be: $d(xS_n,yS_n) = \\ |\{i \ | \ x_i \neq y_i\}|$.
\end{definition}

We will simply write $d(x,y) = d(xS_n,yS_n)$ for representatives $x$ and $y$. We remark that the Hamming distance makes $G(r,d,n)/S_n$ a $G(r,d,n)$-invariant metric space.

\begin{definition}
If $(X, d)$ is a finite metric space the operator $\Delta_k$ is defined as:
$$\Delta_k: C(X) \rightarrow C(X)$$
$$f \mapsto (x \mapsto \sum_{\{y \ | \ d(x,y) = k\}} f(y)).$$
We call $\Delta_1$ the Laplace operator, and it acts as summing over the nearest neighbours.
\end{definition}

\begin{theorem}
If $(G(r,d,n)/S_n, d)$ has the structure of a $G(r,d,n)$-invariant metric space, the zonal spherical functions $\omega_i$ are eigenfunctions of the differential operators $\Delta_k$. The eigenvalues are $\sum_{\{y|d(e,y)=k\}}\omega_i(y)$.
\end{theorem}

\begin{proof}
We turn a right-$S_n$-invariant function $f$ on $G(r,d,n)$ into a function on $G(r,d,n)/S_n$ by setting $\tilde{f}(gS_n) = f(g)$. First we prove $\Delta_k$ commutes with every $g \in G(r,d,n)$. Let $g \in G(r,d,n)$, $f \in C(G(r,d,n)/S_n)$, then:
\begin{flalign*}
& g(\Delta_k f)(x) = (\Delta_k f)(g^{-1}x) = \sum_{\{y \ | \ d(y, g^{-1}x) = k\}} f(y) = \sum_{\{y \ | \ d(gy, x) = k\}} f(y) =
\\ & \sum_{\{z \ | \ d(z, x) = k\}} f(g^{-1}z) = \sum_{\{z \ | \ d(z, x) = k\}} gf(z) = \Delta_k(gf)(x).
\end{flalign*}
If $\omega_i$ is a spherical function, $\omega_i \in C(S_n \backslash G(r,d,n) / S_n) \subseteq C(G(r,d,n) / S_n) = \bigoplus_{i=1}^s \phi_i(V_i)$. If $h \in S_n$:
$$ h(\Delta_k \widetilde{\omega_i}) = \Delta_k (h \widetilde{\omega_i}) = \Delta_k \widetilde{\omega_i}.$$
So $\Delta_k \widetilde{\omega_i} \in C(S_n \backslash G(r,d,n) / S_n)$. We evaluate $\Delta_k \widetilde{\omega_i} (x)$, choose a $g = (g_1, \dots, g_n, id)$ such that $x = gS_n$:
$$\Delta_k \widetilde{\omega_i}(x) = \Delta_k \widetilde{\omega_i}(geS_n) = g^{-1} \Delta_k \widetilde{\omega_i}(eS_n) = \Delta_k g^{-1} \widetilde{\omega_i}(eS_n) = \sum_{\{y | d(e,y) = k\}} g^{-1} \widetilde{\omega_i}(y).$$
For each $y$ we can choose a $y' = (y_1, \dots, y_n, id)$ such that $y = y'S_n$. Then $gy' = y'g$. Thus:
\begin{flalign*}
& \Delta_k \widetilde{\omega_i}(x) = \sum_{\{y | d(e,y) = k\}} g^{-1} \widetilde{\omega_i}(y) = \sum_{\{y | d(e,y) = k\}} \widetilde{\omega_i}(gy) = \sum_{\{y | d(e,y) = k\}}\overline{\rho^i_{11}}(gy')
\\ & = \sum_{\{y | d(e,y) = k\}}\overline{\rho^i_{11}}(y'g) = \sum_{\{y | d(e,y) = k\}}\sum_j \overline{\rho^i_{1j}}(y')\overline{\rho^i_{j1}}(g) = \sum_j (\sum_{\{y | d(e,y) = k\}}\overline{\rho^i_{1j}}(y'))\overline{\rho^i_{j1}}(g)
\end{flalign*}
Recall the function $\phi_i$ is from the section on zonal spherical functions. We obtain:
$$\Delta_k \omega = \sum_j (\sum_{\{y | d(e,y) = k\}}\overline{\rho^i_{1j}}(y'))\widetilde{\phi_i(v_j^i)}.$$
Remember that the $(\widetilde{\phi_i(v_j^i)})_{i,j}$ form a basis for $C(G(r,d,n)/S_n)$ and the $(\widetilde{\phi_i(v_1^i)})_i$ a basis for $C(S_n \backslash G(r,d,n) / S_n)$. Because $\Delta_k \widetilde{\omega_i} \in C(S_n \backslash G(r,d,n) / S_n) \subseteq C(G(r,d,n)/S_n)$ we deduce by comparison of bases that the terms before $\widetilde{\phi_i(v_j^i)}$ must be 0 if $j \neq 1$.
\end{proof}

We will only consider the group $G(r,1,n)$ in the following.

\begin{corollary}
In case of the Hamming distance and the Laplace operator $\Delta_1$, the eigenvalues of $\omega^{(\mathit{k}_0, \dots, \mathit{k}_{r-1})}$ are $\lambda_{(\mathit{k}_0, \dots, \mathit{k}_{r-1})} = n(r-1) - r \sum_{i=1}^{r-1} \mathit{k}_i$.
\end{corollary}

\begin{proof}
We look at the Hamming distance. The eigenvalue of a spherical function \\ $\omega^{(\mathit{k}_0, \dots, \mathit{k}_{r-1})}$ is $\sum_{\{y|d(e,y)=1\}}\omega^{(\mathit{k}_0, \dots, \mathit{k}_{r-1})}(y)$. The $r$-tuple connected to a $y$ with $d(e,y) = 1$ is $(n-1, 0, \dots, 0 ,\mathit{l}_j = 1, 0, \dots, 0)$ where $j \neq 0$ and a $y$ with such an $r$-tuple connected to it appears $n$ times, only one of the entries of $(y_1, \dots, y_n, id)$ is not $1$. Hence we know that the eigenvalue of $\omega^{(\mathit{k}_0, \dots, \mathit{k}_{r-1})}$ is:
$$ \sum_{\{y|d(e,y)=k\}}\omega^{(\mathit{k}_0, \dots, \mathit{k}_{r-1})}(y) = \sum_{j=1}^{r-1} n \omega^{(\mathit{k}_0, \dots, \mathit{k}_{r-1})}_{(n-1, 0, \dots, 0, \mathit{l}_j = 1, 0, \dots, 0)}.$$
We see that:
\begin{flalign*}
& \omega^{(\mathit{k}_0, \dots, \mathit{k}_{r-1})}_{(n-1, 0, \dots, 0, \mathit{l}_j = 1, 0, \dots, 0)} = F((0, \dots, 0, -\mathit{l}_j = -1, 0, \dots, 0), (-\mathit{k}_1, \dots, -\mathit{k}_{r-1}); -n; \widetilde{\Xi}_{r})
\\ & = 1 - \sum_i \frac{\mathit{k}_i}{n}(1 - \xi^{ij}).
\end{flalign*}
Thus we find that the eigenvalue $\lambda_{(\mathit{k}_0, \dots, \mathit{k}_{r-1})}$ of $\Delta_1$ with $\omega^{(\mathit{k}_0, \dots, \mathit{k}_{r-1})}$ as eigenvector is:
\begin{flalign*}
& \sum_{j=1}^{r-1} n \omega^{(\mathit{k}_0, \dots, \mathit{k}_{r-1})}_{(n-1, 0, \dots, 0, \mathit{l}_j = 1, 0, \dots, 0)} =  n(r-1) - \sum_{i=1}^{r-1}\sum_{j=1}^{r-1}\mathit{k}_i(1- \xi^{ij}) =
\\ & n(r-1) - \sum_{i = 1}^{r-1} \mathit{k}_i (r-1 - \sum_{j=1}^{r-1}\xi^{ij}) = n(r-1) -  \sum_{i = 1}^{r-1} \mathit{k}_i (r-1 - (\frac{\xi^{ri} - 1}{\xi^i - 1} - 1)) =
\\ & n(r-1) -  \sum_{i = 1}^{r-1} \mathit{k}_i (r-1 - (0 -1)) = n(r-1) -  \sum_{i = 1}^{r-1} \mathit{k}_i r = n(r-1) - r \sum_{i=1}^{r-1} \mathit{k}_i.
\end{flalign*}
\end{proof}
We directly obtain the eigenvalues of $\Delta_1$ of the $\omega^{(\mathit{k}_0, \dots, \mathit{k}_{r-1})}$ are not unique in general, hence the Laplace operator does not split the zonal spherical functions into different eigenspaces for the Hamming distance.

Now we investigate what $\Delta_1 \omega = \lambda_{(\mathit{k}_0, \dots, \mathit{k}_{r-1})} \omega$ means for concrete values of the hypergeometric functions. If $x$ is the double coset indexed by $(\mathit{l}_0, \dots, \mathit{l}_{r-1})$ we obtain:
\begin{flalign*}
& \Delta_1 \omega^{((\mathit{k}_0, \dots, \mathit{k}_{r-1})}(x) = \sum_{i=1}^n \sum_{y_i \neq x_i} \omega^{((\mathit{k}_0, \dots, \mathit{k}_{r-1})}(x_1, \dots, x_{i-1}, y_i, x_{i+1}, \dots, x_n, id)
\\ & = \sum_{i=0}^{r-1} \mathit{l}_i \sum_{k \neq i} \omega^{((\mathit{k}_0, \dots, \mathit{k}_{r-1})}_{(\mathit{l}_0, \dots, \mathit{l}_i - 1, \dots, \mathit{l}_k +1, \dots, \mathit{l}_{r-1})}.
\end{flalign*}
This means that (if ${e_i} = (\delta_{ij})_{j=1}^{r-1}$, where $\delta$ the Kronecker delta function):
\begin{flalign*}
& \lambda_{(\mathit{k}_0, \dots, \mathit{k}_{r-1})}F(-\mathit{l}, -\mathit{k}; -n; \widetilde{\Xi}_{r})
\\ & = \sum_{i=1}^{r-1}\mathit{l}_0 F(-(\mathit{l} + e_i),-\mathit{k}; -n; \widetilde{\Xi}_{r}) +
\\ & \sum_{i=1}^{r-1} \mathit{l}_i (F(-(\mathit{l} - e_i),-\mathit{k}; -n; \widetilde{\Xi}_{r}) + \sum_{j=1, j \neq i}^{r-1} F(-(\mathit{l} - e_i + e_j), -\mathit{k}; -n; \widetilde{\Xi}_{r})).
\end{flalign*}
In the case $r = 2$ we obtain:
\begin{flalign*}
& (N - 2n) \ {}_2F_1(- x,-n;-N, 2) =
\\ & (N-x) \ {}_2F_1(- x - 1,-n;-N, 2) + x \ {}_2F_1(-x + 1,-n;-N, 2).
\end{flalign*}
This coincides with a result from \cite[Remark 5.3.3]{cecsil08}. Compare with the special case of the product formula when $r=2$.

\section*{Acknowledgements}
I would like to express my sincere gratitude to professor Hirofumi Yamada of Kumamoto University, who graciously received me when I visited Japan and discussed mathematics with me. I would also like to express my gratitude to professor Hiroshi Mizukawa from the National Defense Academy of Japan, whose articles I read and who travelled to Kumamoto University while I was visiting. Finally, I would like to thank professor Gert Heckman and professor Erik Koelink for their help.


\begin{thebibliography}{9}
\bibitem{akamizu} Akazawa, H., \& Mizukawa, H. (2003). \textit{Orthogonal polynomials arising from the wreath products of a dihedral group with a symmetric group.}  J. Combin. Theory Ser. A 104 (2003), no. 2, 371–380. DOI: http://dx.doi.org/10.1016/j.jcta.2003.09.001
\bibitem{cecsil08} Ceccherini-Silberstein, T., Scarabotti, F., \& Tolli, F. (2008). \textit{Harmonic analysis on finite groups: representation theory, Gelfand pairs and Markov chains} (Vol. 108). Cambridge University Press. DOI: http://dx.doi.org/10.1017/CBO9780511619823
\bibitem{dunkladd} Dunkl, C.F. (1976). \textit{A Krawtchouk polynomial addition theorem and wreath products of symmetric groups.} Indiana Univ. Math. J. 25 (1976), no. 4, 335–358. DOI: https://doi.org/10.1090/pspum/034
\bibitem{dunklpol} Dunkl, C. F. (1979). \textit{Orthogonal functions on some permutation groups.} Relations between combinatorics and other parts of mathematics. (Proc. Sympos. Pure Math., Ohio State Univ., Columbus, Ohio, 1978), (ed. Ray-Chaudhuri, D.K.), pp. 129–147, Proc. Sympos. Pure Math., XXXIV, Amer. Math. Soc., Providence, R.I., 1979. DOI: http://dx.doi.org/10.1090/pspum/034/525324
\bibitem{gelfandorigin} Gelfand, I. M. (1950). \textit{Spherical functions on symmetric Riemannian spaces.} Collected papers. II . (ed. Gindikin, S. G., Guillemin, V. W., Kirillov, A. A., Konstant, B., \& Sternberg, S.), pp. 31-35, Springer Collect. Works Math. Springer, Heidelberg, 2015.
\bibitem{humphreys} Humphreys, J.E. (1990). \textit{Reflection groups and Coxeter groups.} Cambridge Studies in Advanced Mathematics, 29. Cambridge University Press, Cambridge, 1990. DOI: http://dx.doi.org/10.1017/CBO9780511623646
\bibitem{koelink} Koelink, H. T. (2000). \textit{$q$-Krawtchouk polynomials as spherical functions on the Hecke algebra of type $B$.}  Trans. Amer. Math. Soc. 352 (2000), no. 10, 4789–4813. 
\bibitem{mizukawa04} Mizukawa, H. (2004). \textit{Zonal spherical functions on the complex reflection groups and $(n+1,m+1)$-hypergeometric functions.}  Adv. Math. 184 (2004), no. 1, 1–17. DOI: http://dx.doi.org/10.1016/S0001-8708(03)00092-6
\bibitem{rahman79} Rahman, M. (1979). \textit{An elementary proof of Dunkl's addition theorem for Krawtchouk polynomials.} SIAM J. Math. Anal. 10 (1979), no. 2, 438–445.  DOI: http://dx.doi.org/10.1137/0510040
\bibitem{sheptodd} Shephard, G. C., \& Todd, J. A. (1954). \textit{Finite unitary reflection groups.}  Canad. J. Math. 6 (1954), 274–304. DOI: http://dx.doi.org/10.4153/CJM-1954-028-3
\end{thebibliography}
\end{document}